\documentclass[
  a4paper, 
  reqno, 
  oneside, 
  11pt
]{amsart}

\usepackage[utf8]{inputenc}

\usepackage[dvipsnames]{xcolor} % Must come before tikz!
\colorlet{cite}{LimeGreen!50!Green}
\usepackage{tikz}  % for diagrams
\usetikzlibrary{arrows,positioning}
\tikzset{ 
  baseline=-2.3pt,
  text height=1.5ex, text depth=0.25ex,
  >=stealth,
  node distance=2cm,
  mid/.style={fill=white,inner sep=2.5pt},
}

\usepackage{lmodern}
\usepackage{tikz-cd}
\usepackage{amsthm, amssymb, amsfonts}

\usepackage{graphicx,caption,subcaption}
\usepackage{braket}  % For nice sets
\usepackage{microtype}

\usepackage[%
  bookmarks=true,			% show bookmarks bar?
  unicode=true,			% non-Latin characters in Acrobat’s bookmarks
  pdftitle={connected2},		%
  pdfauthor={Elizabeth Gasparim}{},	% 
  pdfkeywords={},	% list of keywords
  colorlinks=true,		% false: boxed links; true: colored links
  linkcolor=Blue,			% color of internal links
  citecolor=cite,		% color of links to bibliography
  filecolor=magenta,		% color of file links
  urlcolor=RoyalBlue			% color of external links
]{hyperref}				% One of the most useful packages I have found.
\usepackage{cleveref}

\newtheoremstyle{mydef}
  {}		% Space above environment
  {}		% Space below environment
  {}		% Body font
  {}		% Indent amount (empty = no indent, \parindent = para indent)
  {\scshape}	% theorem head font
  {. }		% Punctuation after heading
  { }		% Space after heading
  {\thmname{#1}\thmnumber{ #2}\thmnote{ #3}}	% Heading spec

\theoremstyle{plain}	% 'plain' is the default.  The others are 'definition' and 'remark'.
\newtheorem{theorem}{Theorem} % putting [section] on the end here tells latex to number the theorem environment within sections, ie. Theorem 2.3 for the third theorem in section 2.
\newtheorem{lemma}[theorem]{Lemma} % putting [theorem] in the middle here tells latex to use the same counter as the theorem environment.  See the output to see how this works.

\newtheorem*{theorem*}{Theorem}

\theoremstyle{mydef} % Here we have used the custom theoremstyle defined above instead of the usual 'definition' style.

\newtheorem*{conjecture*}{Conjecture}
\theoremstyle{remark}
\newtheorem{remark}[theorem]{Remark}
\newtheorem{notation}[theorem]{Notation}

\DeclareMathOperator{\Pic}{Pic}

\DeclareMathOperator{\End}{End}
\DeclareMathOperator{\SL}{\mathrm{SL}}
\DeclareMathOperator{\rk}{rk}

\newtheorem*{proposition*}{Proposition}
\newtheorem*{notation*}{Notation}
\newtheorem*{lemma*}{Lemma}
\newtheorem*{corollary*}{Corollary}
\theoremstyle{definition}
\theoremstyle{remark}

\newcommand{\ce}{\mathrel{\mathop:}=}

\address{Universidad Cat\'olica del Norte, Chile}
\author{Elizabeth Gasparim}
\email{etgasparim@gmail.com}

\title[Connected moduli of instantons on $S^3\times S^1$ ]{Connected moduli of instantons on $S^3\times S^1$ }
\begin{document}

\begin{abstract} I prove connectedness of the moduli space  $\mathcal M_n$ 
of $SU(2)$ instantons on $S^3\times S^1$ with charge $n$.
\end{abstract}
\maketitle
\tableofcontents

In \cite[Sec.\thinspace 2.1.4]{W} Witten asks whether the moduli spaces of instantons 
on  $S^3 \times S^1$ are connected. The purpose of this short note is to answer this question in 
the affirmative. 
The real 4-manifold $S^3 \times S^1$ admits the structure of a complex surface called a Hopf surface, 
which we denote by $X$. 
Using \cite[Thm.\thinspace 1]{Bu}, the moduli space of $SU(2)$ instantons on $S^3 \times S^1$
of charge $n$ can be identified with the moduli space of stable $SL(2,\mathbb C)$ bundles on $X$ with $c_2=n$.
 We denote by $\mathcal M_n$  the moduli space of  stable rank 2 vector bundles on the 
 classical Hopf surface $X$
 with $c_2=n$ and trivial determinant. 
 The goal here is to prove the 
connectedness of  $\mathcal M_n$. I assume that  all
components contain at least one regular bundle.

 We recall a vector bundle on the Hopf surface  is called {\bf regular} when its restriction to each fibre
  of $\pi\colon X\rightarrow \mathbb P^1$ has 
 automorphism group of  minimal dimension, that is $2$. Regular bundles are generic.
 If there is $x\in \mathbb P^1$ for which $E\vert_{\pi^{-1}(x)}$ has automorphism group of 
 dimension 4, then $E$ is called {\bf irregular}, and we say that $E\vert_{\pi^{-1}(x)}$ is an irregular fibre of $E$.
Each bundle $E$ has at most finitely many irregular fibres, these happen at $x_i\in \mathbb P^1$
such that $E\vert_{\pi^{-1}(x_i)}= L \oplus L $ with $L^2=\mathcal O$.
 
 Sections 1 and 2 are summaries of the basic results in the literature about Hopf surfaces and 
 vector bundles on them, and establish some notation. In section 3, I argue that 
 the dense open subset formed by regular bundles with no jumps is connected,
 having locally trivial fibration:

\[
\begin{tikzcd}[swap]
	T^{2n-1} \arrow[]{r}[] {}
	& S_0^{reg}  \arrow[left]{d}[]{}   \\  
		& U 
\end{tikzcd}
\]
where $ S_0^{reg} \subset \mathcal M_n $ and 
	$U  \subset \mathbb P^{2n+1}$  are inclusions of  open sets.
So, being the closure of a connected space, $\mathcal M_n$ is connected.

A quick summary of the proof is as follows: 
 Let  $\mathcal M_n$ denote the moduli space of stable 
$SL(2,\mathbb C)$ 
vector bundles on $X$
with $c_2=n$. The open dense subset $S_0^{reg} \subset \mathcal M_n$  consisting of 
regular bundles with no jumps 
admits the structure of a locally trivial  fibration $S_0^{reg} \rightarrow U \subset \mathbb P^{2n+1}$
where $U$ is   open and dense in $\mathbb P^{2n+1}$, and whose  fibres are tori $T^{2n-1}$;
therefore $S_0^{reg}$ is connected.  Consequently, $\mathcal M_n$,
being contained in 
the closure of the connected space $S_0^{reg}$, is also connected.

\section{Classical Hopf surfaces}
A classical Hopf surface is an elliptically  fibred surface $\pi\colon X\to \mathbb P^1$ 
which is diffeomorphic to $S^1\times S^3$. In  further details,
a {\bf Hopf surface} is  a quotient of the punctured plane $\mathbb C^2\setminus \{0\}$ 
by the action of an infinite cyclic group generated by a contraction of $(\mathbb C^2,0)$. 
When the contraction is multiplication by a diagonal matrix, then the Hopf surface is called {\bf diagonal}.
Hence, a diagonal Hopf surface
is the quotient of $\mathbb{C}^2 \backslash \{ 0 \}$ 
by a cyclic group action generated by
%\mu: \mathbb{C}^2 & \longrightarrow & \mathbb{C}^n \\ 
$$(z_1,  z_2)  \mapsto  (\mu_1 z_1,  \mu_2 z_2)$$
where $\mu_1,  \mu_2 \in \mathbb C$ satisfy 
$0 < | \mu_1 | \leq | \mu_2 | < 1$.

Every diagonal Hopf surface is diffeomorphic to $S^3 \times  S^1$, in particular it is non-K\"ahler
and non-algebraic.
 We consider here  the classical case when the contraction is a multiple of the identity, 
 implying that the surface admits an elliptic fibration.
 
A diagonal Hopf surface is called {\bf classical} when $\mu_1 = \mu_2=\mu$.
Hence, a classical Hopf surface $X$ is generated by an action $(z_1,  z_2)  \mapsto  (\mu z_1,  \mu z_2)$ 
and admits a natural holomorphic elliptic fibration 
\[ \begin{array}{rcl}
\pi\colon X & \rightarrow & \mathbb{P}^1 \\
(z_1,  z_2) & \mapsto & [ z_1 :  z_2],
\end{array} \]
with fibre the elliptic curve $T = \mathbb{C}^\ast/\mu$.\\

\section{Vector bundles on classical Hopf surfaces}\label{rank2}

Given a classical Hopf surface $X$, the Picard group of all holomorphic line bundles on $X$ is given by 
$$\Pic(X) = \frac{H^1(X, \mathcal O)}{H^1(X, \mathbb Z)}= \frac{ \mathbb C}{\mathbb Z}= \mathbb C^*,$$
and since $H^2(X, \mathbb Z)=0$, 
 every line bundle on $X$ has zero first Chern class, hence 
$\Pic^0(X) = \Pic(X)=  \mathbb C^*.$  
The relative Jacobian of a classical Hopf surface $X \stackrel{\pi}{\rightarrow} \mathbb{P}^1$ 
is isomorphic to
\[ J(X) = \mathbb{P}^1 \times T^\ast \stackrel{p_1}{\rightarrow} \mathbb{P}^1, \]
where $T^\ast  \ce \Pic^0(T) $ is the dual elliptic curve determined by a non-canonical identification to $T$.
 Holomorphic vector bundles on elliptically fibred Hopf surfaces were described in \cite{BMo,FMW, Mo1,Mo2}.\\
 
Let $X$ be a classical Hopf surface, $\pi \colon X\to \mathbb P^1$ its elliptic fibration, 
and  let $T$ be the smooth elliptic curve such that all fibres of $f$ are isomorphic to $T$.
%Fix $o\in \mathbb P^1$ and set $C\ce f^{-1}(o)$, hence $C\cong T$, but we may need to fix the isomorphism. 

  A metric on $X$ is called {\bf Gauduchon} if its associated $(1,1)$-form 
  $\omega$ satisfies $\partial \bar\partial \omega = 0$.
 The degree of a line bundle $L$ on $X$ is defined as 
$$ \deg(L) = \int_X F\wedge \omega $$
where $F$ is the curvature of a Hermitian connection on $L$ compatible with $\bar \partial_L$.
For a coherent sheaf $E$, de {\bf degree} is defined as 
$$\deg(E) = \deg\det(E)$$
and the {\bf slope} by 
$$\mu(E) = \frac{\deg(E)}{\rk(E)}.$$
Up to a scalar there is a unique Gauduchon metric on $X$  (we fix it)
 so that each $f^\ast({\mathcal O}_{\mathbb P^1}(t))$ has degree $t$. 
 Stability is then defined as usual, that is, 
 $E$ is {\bf stable} if any subsheaf $F \hookrightarrow E$ satisfies
 $\mu F < \mu E$. \\
 
 We will focus on $\SL(2, \mathbb C)$ bundles, therefore 
 $\Lambda^2 E = \mathcal O$ and $\deg E = 0$.

\begin{notation*} Let ${\mathcal  M}_n$ denote  the moduli space of  stable
  holomorphic vector  bundles $E$ on $X$ with rank $2$, 
trivial determinant  
and $c_2(E) =n$.  
\end{notation*}

By \cite{Bu} we know that the moduli space  $\mathcal M_n$ is diffeomorphic with 
the moduli of $\SL(2, \mathbb C)$ instantons  on $X$, and  \cite[Prop. 3.4.4]{BH} proved that
 ${\mathcal  M}_n$ is a complex smooth manifold of dimension $4n$ ($8n$ real dimensions).\\
 
We want to argue that there exists a non-empty and dense  open subset 
 $\mathcal M_n^{reg}$ of ${\mathcal  M}_n$ such that for each $E\in \mathcal M_n^{reg}$ 
 its restriction $\End(E\vert_{\pi^{-1}(x)})$ 
 has minimal dimension (i.e. 2) for all fibres of $\pi$, that is, for all  $x\in \mathbb P^1$. \\

   We recall a few essential facts about bundles on $X$. Since $X$ is elliptically fibred, 
   we start with the description 
   of vector bundles on an elliptic curve. 
  Atiyah \cite{At} proved that an $SL(2,\mathbb C)$-bundle over an elliptic curve $T$
  is of one of three possible types: 
  \begin{itemize}
  \item[$\iota)$] $L_0\oplus L_0^*$, with $L_0\in \Pic^0(T)$.
  \item[$\iota\iota)$] Non-trivial extensions $0\rightarrow L_0\rightarrow E \rightarrow L_0$, with $L_0^2=\mathcal O$.
  \item[$\iota\iota\iota)$] $L\oplus L^*$, with $L\in \Pic^{k}(T)$,  $k<0$.
  \end{itemize}
  \vspace{3mm}

If   $E$ is an $SL(2,\mathbb C)$-bundle over the Hopf surface $\pi\colon X\rightarrow \mathbb P^1$, 
 then  its restriction to a generic fibre of $\pi$ is either of type $\iota)$ or or type  $\iota\iota)$. Furthermore, 
  the restriction of $E$ is of type $\iota\iota\iota)$ at finitely many fibres \cite[Prop\thinspace3.2.2]{BH}; 
  named the {\bf jumping fibres} of $E$. Hence a jumping fibre $T_x$ of $E$ is one where 
  $E\vert_{T_x} \simeq L\oplus L^*$ with $\deg(L) \neq 0$. \\
   
   To identify those points $x\in \mathbb P^1$ where $E\vert_{\pi^{-1}(x)}$ jumps, it is useful to pick an 
   auxiliary line bundle $V$ over $X$ such that $h^0(\pi^{-1}(x), V^*E) = 0$ for some $x$, and thus
   for generic $x$. Consequently, $\pi_*(V^*E) =0$ and $R^1\pi_*(V^*E)$ is a skyscraper sheaf 
   supported at $n$ points counted with multiplicity, such that, if $h$ is the positive generator 
  of $H^2(\mathbb P^1, \mathbb Z) $ then $\mbox{ch}(R^1\pi_*(V^*E)) = nh$. To define the 
  {\bf multiplicity} of a point $z=0$ one takes a local resolution 
  \begin{equation}\label{multi}
  0 \longrightarrow \mathcal O^m \stackrel{\,f(z)\,}{\longrightarrow} \mathcal O^m \longrightarrow R^1\pi_*(V^*E) \longrightarrow 0\end{equation}
  and the multiplicity is that of the zero of $\det(f(z))$ at $z=0$.

\section{Connectedness  of the moduli space $\mathcal M_n$}

 We consider a stratification of  the moduli space $\mathcal M_n$  according to the number of jumps
counted with multiplicities,
$$\mathcal M_n= \cup_{i=0}^n S_i$$
 %\subsection{The generic stratum} 
where 
$S_0 $  is the generic stratum formed by bundles with zero jumps, and 
$S_i$ contains $i$  jumps counted with  multiplicities. We then have 
$\dim S_0 = 4n$ and $S_0$ is open and dense in $\mathcal M_n$. 
Being the generic stratum, $S_0$ is  the most relevant subset  for the argument 
we intend to make here. 
Observe that  having a jump is a closed condition, therefore 
$\cup_{i=1}^n S_i$ is a closed subset of $\mathcal M_n$. 
Observe that  having a jump is a  closed condition, therefore 
$\cup_{i=1}^n S_i$ is a closed subset of $\mathcal M_n$.
Furthermore, since the support of $R^1\pi_*(V^*E)$ consists of finitely many points, 
from the local resolution defining the multiplicity of a jump, 
we see that bundles with jumps can be expressed as limits of bundles without jumps.
Thus, we have $\overline{S_0} = \mathcal M_n$, and therefore, to show that $\mathcal M_n$ is connected
it suffices to prove that $S_0$ is connected. 

\begin{remark} We observe that here by $\overline{S_0}$ 
we mean the closure inside $\mathcal M_n$, 
sometimes referred to as a partial closure. One could 
obtain a larger closure by adding semistable bundles, for  instance, if one is
attempting to compactify $\mathcal M_n$. However,   we do not consider 
here semistable bundles which are not stable.
\end{remark}

We will use the fact that to each element of $S_0$ 
there corresponds the graph of a rational map $\mathbb P^1\rightarrow \mathbb P^1$ 
of degree  $n$. Indeed, \cite[Sec.\thinspace 3]{BH} associated to each  bundle $E \in \mathcal M_n$ 
 a divisor in the linear system $|\mathcal O(n,1)|$ over $\mathbb P^1\times \mathbb P^1$,
defined as follows.

Given a rank 2 bundle $E$ on $X$, let $V$ be the universal Poincaré line bundle on
$X \times \mathbb C^*$ and consider the first derived image $R^1\pi'_*(E \otimes V)$ where
$\pi'$  the projection $\pi'= \pi\times id \colon X\times \mathbb C^* \to \mathbb P^1 \times \mathbb C^*$. 
The sheaf $R^1\pi'_*(E \otimes V)$ is supported on a divisor on $ \mathbb P^1\times \mathbb C^*$, 
which descends to a divisor 
$$D \subset \mathbb P^1\times \mathbb P^1 = \pi(X) \times \Pic^0(T)/\pm1.$$ When $c_2(E)=n$, the divisor
 $G(E) \ce D $ belongs to the linear system $|\mathcal O(n,1)|$ over $\mathbb P^1\times \mathbb P^1$
 and is called the {\bf graph} of $E$.
 
 Note that here there is a 2-sheeted map $\Pic^0(T) \rightarrow \mathbb P^1$, whose 
 branch points occur at the half-periods (i.e. $ L\oplus L,$ with $L \in \Pic^0(T)$ and  $L^2=0$).

If a bundle has $k$ jumps, then
its graph takes the  form $$\sum_{i=1}^k (x_i \times\mathbb P^1) + Gr(F)$$ where $x_i$ 
are the locations of the jumping lines of $X$ and $F\colon \mathbb P^1\rightarrow \mathbb P^1$ is a
rational map of degree $n-k$.  \cite{BH} also showed that if this divisor includes the 
graph of a non-constant rational map, then the bundle $E$ is stable. In particular, 
this implies:

\begin{remark} $SL(2,\mathbb C)$ bundles on $X$  with $c_2>0$  and without jumps are stable. 
\end{remark}

For $n=1$ the image of $G$ is $\mathbb P^3\setminus (\mathbb P^1\times I)$, and for
$n>1$ the map $G$ is surjective \cite[Thm.\thinspace 5.2.2]{BH}.
Therefore $G(\mathcal M_n)$ accounts for $2n+1$ dimensions, and 
 the complementary $2n-1$ dimensions 
that fill out  this generic stratum are formed by the fibres of the map $G$.
 
 Consider a bundle $E \in \mathcal M_n$ that has no jumps.  In this case, the image $G(E)$ is the graph of a 
 rational (equivalently, holomorphic)  map $\mathbb P^1 \rightarrow \mathbb P^1$ of degree $n$. 
 To each  such  bundle $E$  on the Hopf surface $X$, there corresponds a {\bf spectral curve} 
 $S \subset \mathbb P^1 \times T^*$, that  is, a double cover of the graph $S\rightarrow G(E)$ invariant 
 under the involution,
  and a line bundle on $S$.  Moreover, \cite[Sec.\thinspace 3]{Mo1}  
  proves that $E$ is completely determined by these two objects,
  provided that $E$ is regular.
 Furthermore, 
  the spectral curve $S$ associated to a bundle $E$  with $c_2(E)=n$ is a double cover of a smooth divisor
  $D \simeq \mathbb P^1$ with $4n$ branch points. 
   Such a curve $S$ is given by an equation $y^2 = F(x)$, where $(x, y)$ are local coordinates in 
   $\mathbb P^1 \times T^*$; and the Riemann--Hurwitz formula implies that $S$ is 
   a compact Riemann surface of genus $2n -1$. Thus, if $n = 1$, the curve $S$ is elliptic, 
   and if $n \geq  2$, it is hyperelliptic.
   \cite{Mo1} also shows that the fibre of the graph map over a  generic bundle $E \in \mathcal M_n$
is isomorphic to the Jacobian variety of its corresponding spectral curve $S$. (In such statement, a bundle is 
called generic when it is  regular on every fibre of $X \mapsto \mathbb P^1$; hence when its
 automorphism group has minimal dimension, that is $2$, on every fibre.)

\begin{notation} We  denote by  $S_0^{reg}$ the subset of $\mathcal M_n$
formed by bundles without jumps which in addition are regular.
\end{notation}

Observe that  irregular bundles   in $S_0$ happen precisely when 
the graph is smooth but the spectral curve is singular, and this is a closed condition,
cut out locally by the condition $\det(f)=0$ as in \eqref{multi}.
Since the condition of a bundle being irregular is defined in codimension 1, $S_0^{reg}$ is an open  subset 
of $\mathcal M_n$, therefore if $S_0^{reg}$ is connected, so is $\mathcal M_n$. 

\begin{remark} By analogy with the principles of geometric invariant theory,
we assume that each irreducible component of the moduli space must contain 
at least one regular bundle (by definition).  
Any component formed entirely of irregular bundles, if it were to exist, would not form part of $\mathcal M_n$.
\end{remark} 

To prove connectedness of $S_0^{reg}$,   we
 use the fact that for each bundle $E \in S_0$ the graph $G(E)$ is smooth, and 
furthermore if $E \in S_0^{reg}$ then, in addition, its spectral curve is smooth.\\

Now, when restricting to the subset of regular
elements $S_0^{reg}$ of the stratum $S_0$, we then see that the graph map
% (or alternatively the family of spectral curves, which is also parametrized globally  by $\mathbb P^{2n+1}$)
 gives $S_0^{reg}$  the structure of a 
locally trivial fibration, whose fibre over $E$ is the Jacobian $J(S)$ of  the smooth spectral curve of $E$.
Since $S$ has degree $2n-1$, we have $\dim J(S) = 2n-1$, accounting for the complementary 
dimensions  we were looking for. 
In summary, we conclude:

\begin{lemma}\label{le}
The restriction of the graph map to the generic (regular) stratum  
$S_0^{reg} \subset \mathcal M_n$ has the structure of a locally trivial
fibration  $S_0^{reg} \rightarrow G(S_0^{reg}) = U$, where  
$U \subset |\mathcal O(n,1)| = \mathbb P^{2n+1}$ is Zariski open 
 and whose fibre over $G(E)$ is the Jacobian variety $J(S)$ of the smooth spectral curve $S$ of
 the bundle  $E$.
\end{lemma}

\begin{proof}
We have that the subset 	$U$ consisting of all smooth graphs in $|\mathcal O(n,1)|$ 
  is  Zariski open in $\mathbb P^{2n+1}$, therefore it is connected. 	For each regular bundle 
  $E$ without jumps, the fibre over a point $G(E) \in U$ in the fibration 
\[
\begin{tikzcd}[swap]
	T^{2n-1} \arrow[]{r}[] {}
	& S_0^{reg}  \arrow[right]{d}[]{G}  \\
		& U 
\end{tikzcd} 
\]
is the smooth Jacobian $J(S)\simeq T^{2n-1}$ of the spectral curve of $E$. 
\end{proof}

Assuming that
any component contains a regular bundle, we obtain:

\begin{theorem}
$\mathcal M_n$ is connected.
\end{theorem}

\begin{proof} 
Inside $\mathcal M_n$, we have  $ \overline{S_0^{reg}} = \mathcal M_n $. Thus, lemma \ref{le} implies that
$\mathcal M_n$ is connected.
\end{proof}

Note that this locally trivial fibration does not extend to a locally trivial fibration over 
the entire stratum $S_0$ because 
 the spectral curves of irregular bundles are singular. 
A posteriori, knowing that the 
moduli space $\mathcal M_n$ is connected, we conclude that the subset of all 
irregular bundles forms a divisor in $\mathcal M_n$, which I propose we call the {\bf irregular divisor}.

\section{Acknowledgments} I am very grateful to Edward Witten for enlightening discussions 
about instantons on $S^3\times S^1$. I am also grateful to Edoardo Ballico for providing many 
examples of irregular bundles on Hopf surfaces (details of their existence will appear 
soon in  \cite{BG}), and to Thomas Baird whose suggestions helped me improve 
 the preliminary version of this text. 
 I am a Senior Associate of the Abdus Salam International Centre for Theoretical Physics, Trieste.

\end{document}